\documentclass{amsart}
\usepackage[utf8]{inputenc}
\usepackage[lite]{amsrefs}
\usepackage{amssymb, mathrsfs, comment}
\usepackage{stmaryrd}
\usepackage{enumitem}
\usepackage[english]{babel}
\usepackage{amstext} 
\usepackage{array}
\usepackage{longtable}

\newcolumntype{L}{>{$}c<{$}}

\usepackage{graphicx}

\DeclareMathOperator{\Aut}{Aut}
\DeclareMathOperator{\Ann}{Ann}

\DeclareMathOperator{\GL}{GL}
\DeclareMathOperator{\GSp}{GSp}
\DeclareMathOperator{\Sp}{Sp}

\DeclareMathOperator{\pr}{pr}
\DeclareMathOperator{\id}{id}
\DeclareMathOperator{\vol}{vol}

\DeclareMathOperator{\Spec}{Spec}

\newcommand{\brZ}{\breve{\mathbb{Z}}_p}
\newcommand{\brQ}{\breve{\mathbb{Q}}_p}

\numberwithin{equation}{section}
\theoremstyle{plain} 
\newtheorem{thm}[equation]{Theorem}
\newtheorem{cor}[equation]{Corollary}
\newtheorem{lemma}[equation]{Lemma}
\newtheorem{prop}[equation]{Proposition}

\theoremstyle{definition}
\newtheorem{defn}[equation]{Definition}

\theoremstyle{remark}
\newtheorem{rem}[equation]{Remark}
\newtheorem{claim}[equation]{Claim}
\newtheorem{ex}[equation]{Example}

\title[Automorphisms of generic supersingular abelian varieties]{Oort's conjecture on automorphisms of generic supersingular abelian varieties}

\author{Eva Viehmann}
\address{Universit\"at M\"unster, Einsteinstr. 62, 48149 M\"unster, Germany}
\thanks{I thank Valentijn Karemaker, Pol van Hoften, Paul Philippe, Fabian Schnelle and Chia-Fu Yu for interesting and helpful discussions. During this work, the author was partially supported by the DFG through CRC TRR 326 GAUS, CRC 1442 ``Geometry: Deformations and Rigidity'', under Germany's Excellence Strategy EXC 2044/2 –
390685587, Mathematics Münster: Dynamics–Geometry–Structure and by a Leibniz prize.}

\begin{document}

\begin{abstract}
We prove Oort's conjecture that generically on the supersingular locus of the moduli space of principally polarized abelian varieties of genus $g$ and in characteristic $p$, the automorphism group of the universal principally polarized abelian variety consists only of $\pm 1$, unless $g=2, 3$ or $4$ and $p=2$. On the way, we provide an explicit description of the $a=1$ locus in the Rapoport-Zink space of principally polarized supersingular $p$-divisible groups of any dimension $g$. We also prove analogous results for generic automorphism groups on moduli spaces of supersingular $p$-divisible groups with and without polarization.
\end{abstract}

\maketitle


\section{Introduction}

Let $g\geq 1$ be an integer and let $p$ be prime. Denote by  $\mathscr A_g$ the moduli space of principally polarized abelian varieties of dimension $g$ over $\overline{\mathbb{F}}_p$ with level $N$-structure for some $N$ coprime to $p$. A central tool to study $\mathscr A_g$ is to decompose it according to the isogeny class of the $p$-divisible group of the universal abelian variety. In this way, one obtains a stratification into finitely many locally closed subschemes, called Newton strata. Let $\mathscr S_g$ denote the unique closed Newton stratum. Let $k$ be an algebraically closed field of characteristic $p$ and let $(A_x,\lambda_x)$ be the principally polarised abelian variety corresponding to a $k$-valued point $x$ of $\mathscr A_g$. Then $x$ factors through $\mathscr S_g$ if and only if  $A_x[p^{\infty}]$ is isogenous to a product of $g$ copies of the $p$-divisible group of a supersingular elliptic curve.

Chai and Oort \cite{ChaiOort} show that generically on each of the non-supersingular Newton strata in $\mathscr A_g$, we have Aut$(A_x,\lambda_x)=\{\pm 1\}$. The key tool of their proof is to study automorphisms on central leaves in these Newton strata, i.e. on the locus where the $p$-divisible group of the universal abelian variety is isomorphic to a fixed given one. A similar strategy cannot be applied to the supersingular locus where the central leaves are zero-dimensional. Nevertheless, Oort \cite[Problem 4]{EMO} conjectured in 2001 that if $g\geq 2$, also on a dense open subscheme of the supersingular Newton stratum we have Aut$(A_x,\lambda_x)=\{\pm 1\}$. In the 25 years since then, many people have studied this question.

\begin{itemize}
\item For $g=2$ and $p>2$ the conjecture has been shown by Ibukiyama \cite{Ibukiyama} and independently and with different methods by Karemaker and Pries \cite[ Prop.~7.6]{KaremakerPries}. Ibukiyama also shows that the conjecture is false if $g=p=2$.
\item For $g=3$, Karemaker, Yobuko and Yu \cite{KYY} consider the decomposition of the supersingular locus according to the size of Aut$(A_x,\lambda_x)$. They show that generically, the automorphism group is $\{\pm 1\}$ if $p>2$, but consists of 8 elements if $p=2$. 
\item For $g=4$ and $p>2$, the conjecture is proved by Karemaker and Yu in \cite{KaremakerYu}, building on a description of the supersingular locus by Harashita \cite{HarPpav4}. An independent proof is also given by Dragutinovi\'c in \cite{Dragutinovic}. Originally, Karemaker-Yu also claimed to have shown the conjecture for $g=4$ and $p=2$. Recently, they added a new example showing that in this case, the conjecture is in fact also false.
\item Karemaker and Yu \cite{KaremakerYu} prove the conjecture for all even $g$ and all $p\geq 5$.
\end{itemize}

The goal of this work is to prove Oort's conjecture in all remaining cases including those where $p=2$.

\begin{thm}[Oort's conjecture]\label{thmmain}
Let $g\geq 2$ and let $p$ be a prime with $(g,p)\neq (2,2),(3,2), (4,2)$. Then there is a dense open subscheme $Y$ of $\mathscr S_g$ such that for every $x\in Y(\overline{\mathbb{F}}_p)$, we have  ${\rm Aut}(A_x,\lambda_x)=\{\pm 1\}$.
\end{thm}

After the first version of this work was made public, Karemaker and Yu provided a different proof of the conjecture under the assumption that $p>2$, compare \cite{KaremakerYu2026}.

The existing results towards Oort's conjecture use a wide range of different methods,  from explicit calculations of automorphisms of the associated $p$-divisible groups to considerations of Ekedahl-Oort strata and to the construction of abelian varieties with little automorphisms that arise as Jacobians of suitable curves. In all of these cases, it turns out that the conjecture is significantly more difficult to prove if $p=2$ and if $g$ gets larger. This can be seen from the fact that not a single positive case of the conjecture was known for $p=2$ before the present work. For larger $g$, the computational effort and the complexity of the systems of relations increase, and require a more systematic treatment of the resulting conditions. 

Our strategy starts with some reduction steps already used in \cite{KYY} and \cite{KaremakerYu} for $g=3$ and $4$. In order to deal with the additional complexity due to general $g$ and all primes $p$, we then introduce and use a new description of the generic polarized supersingular $p$-divisible groups. In previous work, the descriptions were based on so-called polarized flag type quotients, a notion introduced by Li and Oort \cite{LiOort} to parametrize supersingular abelian varieties. Instead, we use Rapoport-Zink's uniformization theorem to relate the problem to certain Rapoport-Zink moduli spaces of polarized supersingular $p$-divisible groups. In Section 2 we explain the reduction strategy and several general results on the geometry of these moduli spaces, in particular their sets of irreducible components.

Generically on the involved Rapoport-Zink-spaces, the $a$-invariant of the universal $p$-divisible group is $1$. In Section \ref{seca1} we provide the main new tool that we later use to prove the conjecture. We prove an explicit description of the $a=1$-locus in the Rapoport-Zink space of polarized supersingular $p$-divisible groups of any dimension $g$ by analyzing the conditions that the generators of the associated Dieudonn\'e modules have to satisfy. In all previous proofs of particular cases of Oort's conjecture that used the $a=1$-locus, one had $g\leq 4$, and one could describe $\overline{\mathbb{F}}_p$-valued points of the $a=1$-locus using points of a subscheme of some affine space that is defined by a single Fermat-type equation as well as several open conditions. In contrast to that, for the general case we obtain in Remark \ref{reminitcond} a system of $g-3$ closed conditions that are induced by the polarization. These conditions take the form that certain pairings with values in the rational Witt vectors should take integral values. In Proposition \ref{propcondai} we rearrange these conditions in such a way that we obtain again an iterative description with coordinates in suitable subschemes of affine spaces subject to some open conditions and in total $|\{(l,j)\mid l\geq 0, 1\leq j<g-1-l, l+j\equiv g\pmod{2}\}|$ closed equations which are then again polynomials with values in $\overline{\mathbb{F}}_p$ instead of in the rational Witt vectors.

In Section \ref{secpfoort} we use this description to study automorphisms of supersingular $p$-divisible groups with $a=1$ and prove Theorem \ref{thmmain}. Here, we use an iterative approach, first considering all primes $p\geq 5$, then $p=3$, and finally $p=2$. The results of Section \ref{seca1} allow us to handle all $g\geq 5$ at once.

In Section \ref{secnonpol} we extend our results of automorphism groups of generic supersingular $p$-divisible groups to the non-polarized case. This corresponds to studying the basic locus of a unitary Shimura variety of signature $(g,g)$ at a split prime $p$, i.e.~such that the associated group over $\mathbb{Q}_p$ is $\GL_{2g}$. We prove that generically on the supersingular locus, every automorphism of finite order of the associated $p$-divisible group is a unit root in $\mathbb{Z}_p^{\times}$.

\section{Reformulation}\label{secred}

In this section we review several results related to moduli spaces of principally polarized abelian varieties and $p$-divisible groups and apply them to reduce the proof of Theorem \ref{thmmain} to a statement about  self-dual Dieudonn\'e lattices of $a$-invariant 1.

\begin{defn}
Let $(A,\lambda)$ be an abelian variety over a field $k$ of characteristic $p$ together with a principal polarization. Then 
$$\Aut(A,\lambda):=\{\varphi:A\overset{\cong}{\longrightarrow} A\mid \lambda\circ \varphi=\varphi^{\vee}\circ \lambda\}.$$
\end{defn}

First recall  that the automorphism group of any principally polarized abelian variety is finite. In particular, in order to prove Theorem \ref{thmmain}, it is enough to show that generically on $\mathscr{S}_g$, the $p$-divisible group of the universal principally polarized abelian variety does not have any automorphisms of finite order other that $\pm 1$.

Furthermore, let $(A,\lambda)$ be a principally polarized abelian variety over $\overline {\mathbb{F}}_p$ and let $(X,\lambda)$ be its $p$-divisible group with the induced polarization. Then the natural map $\Aut(A,\lambda)\rightarrow\Aut (X,\lambda)$ is injective. In \cite{RZ}, Rapoport and Zink define and study moduli spaces of $p$-divisible groups in a very general context of which we need the following particular case. There is a unique isomorphism class of pairs consisting of a superspecial $p$-divisible group of dimension $g$ together with a principal polarization. Let $(\mathbb X,\lambda_0)$ be such a pair. We consider the functor assigning to any scheme $S$ such that $p$ is locally nilpotent on $S$ the set 
$$\mathcal M_g(S):=\{(X,\lambda,\rho)\}/\sim$$ where $X$ is a $p$-divisible group over $S$, where $\lambda$ is a $p$-power multiple of a principal polarization of $X$ and where $\rho: \mathbb{X}_{\bar S}\rightarrow X_{\bar S}$ is a quasi-isogeny over the reduction modulo $p$ of $S$ that is compatible with the polarizations. It is represented by a formal scheme $\mathcal M_g$ locally formally of finite type over Spf $\brZ$ where $\breve{\mathbb{Z}}_p=W(\overline{\mathbb{F}}_p)$. The uniformization theorem \cite[Thm.~6.30]{RZ} then implies that Theorem \ref{thmmain} follows from the following slightly stronger theorem.

\begin{thm}\label{thmgenautRZ}
There is an open and dense subscheme $Y\subseteq \mathcal M_g$, such that for any $y\in Y$ the reduction $(X_y,\lambda_y)$ of the universal polarized $p$-divisible group $(X,\lambda)$ on $\mathcal M_g$ does not have any automorphisms of finite order other that $\pm 1$. 
\end{thm} 
 To prove Theorem \ref{thmgenautRZ}, it is enough to consider the underlying reduced subscheme of $\mathcal M_g$, a scheme locally of finite type over $\Spec \overline{\mathbb{F}}_p$. By abuse of notation we will from now on refer to this reduced scheme by $\mathcal M_g$.

The $a$-number of a $p$-divisible group $X$ over an algebraically closed field of characteristic $p$ is defined as $\dim {\rm {Hom}} (\alpha_p,X)$. Let $\mathcal M_{g}^{\circ}$ be the locus in $\mathcal M_{g}$ where the $a$-number of the universal $p$-divisible group is $1$. By \cite[Lemma 3.1]{polpdiv}, $\mathcal M_{g}^{\circ}$ is open and dense in $\mathcal M_{g}$. Thus it is enough to study the automorphism group of the universal polarized $p$-divisible group $(X,\lambda)$ over $\mathcal M_{g}^{\circ}$.

Recall the equivalence between $p$-divisible groups over $\bar{\mathbb F}_p$ 
and their Dieudonn\'e modules. Let $\brZ =W(\overline{\mathbb{F}}_p)$ and let $\breve{\mathbb{Q}}_p$ be its field of fractions. Let $\sigma$ denote the Frobenius morphism of $\overline{\mathbb{F}}_p$ over $\mathbb{F}_p$ and also of $\breve{\mathbb{Q}}_p$ over $\mathbb{Q}_p$. For an element $a$ we use $a^{\sigma}:=\sigma(a)$.  In our case, the Dieudonn\'e module of a triple $(X,\lambda,\rho)$ associated with some $x\in \mathcal M_g(\overline{\mathbb{F}}_p)$ is a triple $(M,F,\langle\cdot,\cdot\rangle)$ where $M$ is a $2g$-dimensional free $\brZ$-module, $F:M\rightarrow M$ is $\sigma$-linear with $M\supseteq F(M)\supseteq pM$ and $\langle \cdot,\cdot\rangle$ is a symplectic pairing on $M\otimes_{\brZ}\brQ$ such that $M^{\vee}=c\cdot M$ for some $c\in \brQ$. Here, $M^{\vee}=\{v\in M\otimes_{\brZ}\brQ\mid \langle v,M\rangle\subseteq \brZ\}.$ Furthermore, $F$ preserves the symplectic pairing up to a scalar. We have 
\begin{align*}
\Aut (X,\lambda)&=\Aut (M,F,\langle\cdot,\cdot\rangle)\\
&=\{j:M\rightarrow M\mid j\circ F=F\circ j, \langle jv,jw\rangle=\langle v,w\rangle \text{ for all }v,w\in M\}.
\end{align*}
As usual, the Verschiebung on $M$ is defined as $V:=pF^{-1}$. One then has $a(X)=\dim_{k} M/(FM+VM)$, and for any Dieudonn\'e module $M$, the corresponding quantity $a(M):=\dim_{k} M/(FM+VM)$ is also called the $a$-number of $M$. Let 
$$\mathscr D=\brZ[F,V]/(\sigma(a)F-Fa, aV-V\sigma(a), FV-p,VF-p)$$ 
be the Dieudonn\'e ring for $\overline{\mathbb F}_p$. Then $a(M)=1$ if and only if there is a $v\in M$ with $\mathscr D\cdot v=M$.

Let $(N,F,\langle\cdot,\cdot\rangle)$ be the rational Dieudonn\'e module of $(\mathbb{X},\lambda_0)$. Then $N\cong \mathbb{Q}_p^{2g}$. Since $\mathbb{X}$ is superspecial, there is a basis $e_1,\dotsc, e_g, f_1, \dotsc, f_g$ of $N$ such that $F(e_i)=f_i$, $F(f_i)=pe_i$ and  
\begin{equation}\label{eqdefpair}
\langle e_i,f_{g+1-i}\rangle=- \langle f_i,e_{g+1-i}\rangle=1
\end{equation}
for all $i$, and all other pairings among standard basis vectors vanish. Let $\breve N=N\otimes_{\mathbb{Q}_p}\breve{\mathbb{Q}}_p$. Then any $M$ as above is identified via $\rho$ with a sublattice of $\breve N$, compatible with the induced Frobenius morphisms $F$ and symplectic pairings. In particular, we have 
\begin{align*}
\Aut (M,F,\langle\cdot,\cdot\rangle)&\subset \Aut (\breve N,F,\langle\cdot,\cdot\rangle )\\
&=\{j:\breve N\rightarrow \breve N\mid j\circ F=F\circ j, \langle jv,jw\rangle=\langle v,w\rangle \text{ for all }v,w\in \breve N\}\\
&=: J(\mathbb{Q}_p).
\end{align*}
The above identification of $\breve N$ with $\brQ^{2g}$ induces an isomorphism $J(\mathbb{Q}_p)\cong GSp_g(D)$ where $D$ is the division algebra of invariant $\tfrac{1}{2}$ over $\mathbb Q_p$. Let $\mathcal O_D$ denote a maximal order of $D$ and $\Pi\in \mathcal O_D$ a uniformizer with $\Pi^2=p$. The group $J(\mathbb{Q}_p)$, which can also be interpreted as the group of self-quasi-isogenies of $(\mathbb{X},\lambda_0)$, has a natural action on $\mathcal{M}_g$ and on $\mathcal M_g^{\circ}$ by precomposing $\rho$ with the respective element of $J(\mathbb{Q}_p)$. By \cite[Thm.~2]{polpdiv}, this action is transitive on the set of irreducible components of $\mathcal M_g$ or, equivalently, of $\mathcal M_g^{\circ}$.  For $(X,\lambda,\rho)$, $(X',\lambda',\rho')$ in the same $J(\mathbb Q_p)$-orbit on $\mathcal{M}_g^{\circ}$, the automorphism groups are conjugate subgroups of $J(\mathbb{Q}_p)$ and hence isomorphic. Thus it is enough to show that on some fixed irreducible component of $\mathcal M_g^{\circ}$, the torsion part of the automorphism group of the generic principally polarized $p$-divisible group consists of $\pm 1$.

Let us describe how  $\mathcal M_g^{\circ}$ is subdivided into irreducible components, or equivalently into connected components. Let $\tau: \breve N\rightarrow \breve N$ with $$\tau:=p^{-1}F^2.$$ As a special case of \cite[Lemma 9]{ZinkSlope} we obtain

\begin{lemma}[Zink's lemma]
Let $X$ be a $p$-divisible group over $\overline{\mathbb{F}}_p$ and let $(M,F)$ be its Dieudonn\'e module. Then $M_{\tau}=M+\tau(M)+\dotsm$ is a Dieudonn\'e lattice and the unique smallest $\tau$-stable lattice containing $M$. 
\end{lemma}

\begin{cor}\label{corMtau}
Let $(M,F,\langle\cdot,\cdot\rangle)$ be the Dieudonn\'e module of a principally polarized supersingular $p$-divisible group and let $M_{\tau}$ be as in the previous lemma with the induced Frobenius and symplectic pairing (with values in $\breve{\mathbb{Q}}_p$). Then we have a natural injection $\Aut(M,F,\langle\cdot,\cdot\rangle)\subseteq \Aut(M_{\tau},F,\langle\cdot,\cdot\rangle)$.
\end{cor}
The lattice $M_{\tau}$ does not need to be self-dual. The set of automorphisms on the right is defined as those automorphisms of $(M_{\tau},F)$ that preserve the pairing $\langle\cdot,\cdot\rangle$ (with values in $\brQ$).
\begin{proof}
This follows from the uniqueness of $M_{\tau}$ together with the fact that any automorphism of $(M, F,\langle\cdot,\cdot\rangle)$ commutes with $F$ and hence with $\tau$.
\end{proof}

\begin{prop}\label{propsetirrcomp}
Let $S(\tau)$ be the set of $\tau$-stable Dieudonn\'e lattices in $N$, viewed as a discrete set. Then 
\begin{align*}
s:\mathcal M_g^{\circ}(\overline{\mathbb{F}}_p)&\rightarrow S(\tau)\\
M&\mapsto M_{\tau}
\end{align*}
defines a locally constant and $J(\mathbb{Q}_p)$-equivariant morphism of schemes $\mathcal M_g^{\circ}\rightarrow S(\tau)$. The image consists of a single $J(\mathbb Q_p)$-orbit. For any fixed lattice $\Lambda$ in the image, $\mathcal C_{\Lambda}:=s^{-1}(\{\Lambda_0\})$ is an irreducible component of $\mathcal M_g^{\circ}$.
\end{prop}

\begin{proof}
The first two assertions are shown in \cite[Section 6.1]{polpdiv}. By \cite[Thm.~2]{polpdiv}, $J(\mathbb{Q}_p)$ acts transitively on the set of irreducible components of $\mathcal M_g^{\circ}$. The stabilizer of any given irreducible component is contained in the stabilizer of the associated element $\Lambda_0\in S(\tau)$, a maximal parahoric subgroup of $J(\mathbb{Q}_p)$.  By \cite[Thm.~A]{HeZhouZhu}, the stabilizer of any irreducible component is a maximal parahoric subgroup, and thus has to agree with the stabilizer of $\Lambda_0$. Hence $s^{-1}(\Lambda_0)$ is irreducible.
\end{proof}

In Proposition \ref{proplambda0}, we describe explicitly one of the lattices in the image of $s$. 

\begin{rem}
The lattice $M_{\tau}$ is also related to the notion of flag type quotients that is frequently used to parametrize the $a=1$-locus of $\mathcal M_g$ for small $g$. We do not use this notion in the course of this work, but let us briefly describe the relation for the convenience of readers who are more familiar with this other notion. Let $M\subseteq \breve N$ be the Dieudonn\'e lattice corresponding to some $x\in \mathcal M_{g}^{\circ}(\overline{\mathbb{F}}_p)$ and set for $ i\leq g-1$ $$M(i):= M+F^{g-1-i}M_{\tau}= M+V^{g-1-i}M_{\tau}.$$ Using \cite[Section 4]{modpdiv}, compare Lemma \ref{lem33} below, one can show that $F^{g-1}M_{\tau}\subseteq M\subseteq M_{\tau}$ and that $F^{g-1}M_{\tau}$ is the largest $\tau$-stable lattice contained in $M$. Thus $M(i)=M$ for all $i\leq 0$, $M(g-1)=M_{\tau}$, and each $M(i)/M(i-1)$ is annihilated by $F$ and by $V$. Our $\tau$-stable lattices are Dieudonn\'e modules of superspecial $p$-divisible groups and the $M(i)$ are the lattices occurring naturally as Dieudonn\'e modules of intermediate steps in flag type quotients. 
\end{rem}

\section{An irreducible component of the Rapoport-Zink space}\label{seca1}

Recall that $J(\mathbb Q_p)$ acts transitively on the set of irreducible components of $\mathcal M_g^{\circ}$ and the automorphism group of the universal polarized $p$-divisible group is constant on each $J(\mathbb Q_p)$-orbit. Thus it is enough to show that for some fixed irreducible component $\mathcal C_{\Lambda_0}$, there is a dense open subscheme on which the polarized $p$-divisible groups do not have automorphisms of finite order other than $\pm 1$. In this section we use the results in \cite{modpdiv} and \cite{polpdiv} to describe in Proposition \ref{proplambda0} one of the lattices in the image of $s$. We determine the stabilizer of that component in $J(\mathbb Q_p)$ and prove in Lemma \ref{propopen} that the locus where the only automorphisms of finite order are $\pm 1$ is open. In particular, for the proof of the main theorem it is enough to show that this locus is non-empty. We then introduce and study more explicit coordinates describing generators of the Dieudonn\'e lattices for the associated irreducible component $\mathcal C_{\Lambda_0}$. For $g=4$, this is related to the description given in \cite{HarPpav4}. We describe conditions for the Dieudonn\'e modules corresponding to geometric points of $\mathcal C_{\Lambda_0}$ that are later used to prove the main theorem. Note that as in \cite{modpdiv} one could use display theory and extend this description to describe the whole component. However, we do not need such a description here.

Let $M$ be a Dieudonn\'e lattice corresponding to a point of some $\mathcal C_{\Lambda_0}$. Since $\dim M/(FM+VM)=1$, the lattice is as Dieudonn\'e module generated by a single element. Let $v\in M$ be such a generator. It is unique up to scaling and up to elements of $FM+VM$. In the rest of this section, we conversely study the conditions for an arbitrary element $v\in \Lambda_0$ to be a generator of some Dieudonn\'e lattice corresponding to a point of $\mathcal C_{\Lambda_0}$.  In \cite{modpdiv} and \cite{polpdiv} the focus was to establish a bijection between the Dieudonn\'e modules of $a$-invariant $1$ and suitable uniquely chosen generators $v$. In this paper, the ambiguity in choosing the generator $v$ of $M$ is not relevant, so we consider all possible generating vectors and rather focus on the relations induced by the conditions that $M_{\tau}$ is a given $\tau$-stable lattice and that $M$ is self-dual. Additional conditions making the choice of $v\in M$ unique are already discussed in \cite{polpdiv}.

Assume at first that $v\in \breve N$ such that $M:=\mathcal D v$ is a lattice. Let $\Lambda_0=M_{\tau}$. It is an $F$- and $\tau$-stable lattice, thus $\Lambda_0/F\Lambda_0$ is a $g$-dimensional $\overline{\mathbb{F}}_p$-vector space for which we can choose $\tau$-stable generators $\bar X_1,\dotsc, \bar X_g$. We choose $\tau$-stable representatives $X_i\in \Lambda_0$ and let $Y_i=FX_i=VX_i$. The $X_i$ and $Y_i$ freely generate $\Lambda_0$ as $\brZ$-module. Moreover, they generate the submodule of $\tau$-stable vectors of $\Lambda_0$ as $\mathbb{Z}_{p^2}$-module.

We write 
\begin{equation}\label{eqcoordv}
v=\sum_{i=1}^g \sum_{l=0}^{\infty} [a_{i,l}]F^l X_i
\end{equation}
with $a_{i,l}\in \overline{\mathbb{F}}_p$. By \cite[Lemma 4.8]{modpdiv}, the condition that $M$ is a lattice with  $M_{\tau}=\Lambda_0$ is equivalent to $[a_{1,0}:\dotsm:a_{g,0}]$ not being contained in any $\mathbb{F}_{p^2}$-rational hyperplane in $\mathbb{P}^{g-1}$. 

\begin{lemma}\label{lemailinindep}
Let $x_1,\dotsc, x_n\in\overline{\mathbb{F}}_p$. Then $[x_1:\dotsm: x_n]$ is not contained in any ${\mathbb{F}}_{p^2}$-rational hyperplane if and only if the $n\times n$-matrix $(\sigma^{2j}(x_i))_{i,j}$ with $i,j=1,\dotsc, n$ is invertible. It is also equivalent to $(\sigma^{2j-l}(x_i))_{i,j}$ being invertible for any (or all) fixed $l$.
\end{lemma}
\begin{proof}
Let $x=(x_i)\in \overline{\mathbb{F}}_{p}^n$. The first condition is equivalent to the condition that the $\overline{\mathbb{F}}_{p}$-vector space generated by all $\sigma^{2j}(x)$ (with $j\geq 0$) is equal to $\overline{\mathbb{F}}_{p}^n$. Applying $\sigma^2$, we may equivalently consider the $j\geq 1$. This condition holds if and only if these vectors for $1\leq j\leq n$ are linearly independent. Written up in coordinates, this is precisely the second condition in the lemma. Applying $\sigma^{-l}$, we get the last condition.
\end{proof}

\begin{lemma}\label{lem33}
Let $v$ be as in \eqref{eqcoordv} satisfying the conditions of Lemma \ref{lemailinindep}. Let $M=\mathcal D v$ and let $\Lambda_0=M_{\tau}$ be as above. 
\begin{enumerate}
\item For $i=0,\dotsc, g-1$, the sub-vector space $(M\cap F^{j}\Lambda_0)/(M\cap F^{j+1}\Lambda_0)\subseteq F^j\Lambda_0/F^{j+1}\Lambda_0$ is freely generated by the cosets of the $j+1$ vectors $F^aV^bv$ with $a,b\geq 0$ and $a+b=j$. 
\item $F^{g-1}\Lambda_0=\Pi^{g-1}\Lambda_0=V^{g-1}\Lambda_0\subseteq M$ is the largest $\tau$-stable lattice contained in $M$.
\item Assume that $M$ is self-dual. Then $$\Lambda_0\overset{g\tfrac{g-1}{2}}{\supseteq}M=M^{\vee}\overset{g\tfrac{g-1}{2}}{\supseteq}F^{g-1}\Lambda_0=\Lambda_0^{\vee}.$$
\end{enumerate}
\end{lemma}
Notice that for (1) and (2) we do not need the symplectic pairing, these assertions hold for any supersingular isocrystal $\breve N$.
\begin{proof}
For (1) we use induction on $j$ and thus may assume that $M\cap F^j\Lambda_0$ does not contain any non-trivial linear combination of the $F^aV^bv$ with $a+b<j$. Hence $(M\cap F^{j}\Lambda_0)/(M\cap F^{j+1}\Lambda_0)\subseteq F^j\Lambda_0/F^{j+1}\Lambda_0$ is generated by the residue classes of the $j+1$ vectors $F^aV^bv$ with $a,b\geq 0$ and $a+b=j$. We have to show that they are linearly independent. This, however, is an immediate consequence of Lemma \ref{lemailinindep} since the coordinate vector of the image of $F^aV^bv$ with respect to the basis $F^j X_i$ with $1\leq i\leq g$ is $(a_{i,0}^{\sigma^{a-b}})_i$.

The first assertion of (2) follows from (1). We claim that the sub-vector spaces $(M\cap F^{j}\Lambda_0)/(M\cap F^{j+1}\Lambda_0)$ for $j<g-1$ do not contain any $\tau$-stable element. Indeed, assume that some $w$ is contained in one of these sub-vector spaces. Then $Fw-Vw$ is a non-trivial linear combination of the $F^aV^bv$ with $a+b=j+1\leq g-1$ and hence non-zero. Applying $V^{-1}$ we obtain that $\tau(w)\neq w$. From the claim, (2) follows since any $\tau$-stable lattice is generated by its $\tau$-stable elements.

(3) then follows since $\Lambda_0^{\vee}$ is the largest $\tau$-stable lattice contained in $M^{\vee}$.
\end{proof}

It remains to study the additional condition that $M$ is self-dual. Recall that the Dieudonn\'e module of $(\mathbb{X},\lambda_0)$ is of the form $(M_0,F=b\sigma,\langle\cdot,\cdot\rangle)$ where $M_0$ is a free $\mathbb{Z}_p$-module with generators $e_1,\dotsc, e_g,f_1,\dotsc, f_g$, where $F(e_i)=f_i$, $F(f_i)=pe_i$ and the pairing is given by \eqref{eqdefpair}. For this basis we also identify $\Aut(\mathbb{X},\lambda_0)$ with $\GSp_{g}(\mathcal O_D)$ such that the uniformizer $\Pi$ acts by $\Pi e_i=f_i$ and $\Pi f_i=pe_i$. 

Note that $F,V$ and $\Pi$ are defined using the same values of the base vectors $e_i,f_i$. However, $\Pi$ is extended linearly, $F$ is $\sigma$-linear and $V$ is $\sigma^{-1}$-linear. 

The above choice of coordinates induces the following direct sum decomposition of $\breve N=M_0\otimes_{\mathbb{Z}_p}\brQ$. Let $N_0$ be the span of all $e_i,f_i$ with $i\leq \lfloor \tfrac{g}{2}\rfloor$ and let $N_1$ be the span of all $e_i,f_i$ with $g+1-i\leq \lfloor \tfrac{g}{2}\rfloor$. If $g$ is odd, let $m=(g+1)/2$ the {\it m}iddle index and let $N_{1/2}$ be the span of $e_{m},f_{m} $. Then all $N_j$ are stable under $F$, and 

$$\breve N=\begin{cases}
N_0\oplus N_1&\text{if }g\text{ is even}\\
N_0\oplus N_{1/2}\oplus N_1&\text{if }g\text{ is odd.}
\end{cases}$$
Furthermore, $N_0$ is orthogonal to $N_0\oplus N_{1/2}$ and similarly for $N_1$.

\begin{prop}\label{proplambda0}
Let 
\begin{equation}
\Lambda_0=\begin{cases}
\langle e_1,\dotsc, e_{g/2},\Pi^{-g+1}e_{g/2+1},\Pi^{-g+1}e_g\rangle_{\breve{\mathbb{Z}}_p[\Pi]}&\text{if }g\text{ is even,}\\
\langle e_1,\dotsc, e_{m-1},\Pi^{(-g+1)/2}e_{m},\Pi^{-g+1}e_{m+1},\Pi^{-g+1}e_g\rangle_{\breve{\mathbb{Z}}_p[\Pi]}&\text{if }g\text{ is odd.}
\end{cases}
\end{equation}
Then $\Lambda_0=M_{\tau}$ for some self-dual Dieudonn\'e lattice $M\subseteq \breve N$ of $a$-invariant 1.
\end{prop}
We rename the generators of $\Lambda_0$ as $\breve{\mathbb{Z}}_p$-module as $X_i,Y_i$ with $X_i=\Pi^{c_i}e_i$ for exponents $c_i$ as in the proposition and with $Y_i=\Pi X_i$. 
\begin{proof}
Denote by $\pr_0,\pr_{1/2},\pr_1$ the projections to the corresponding summands $N_i$ of $\breve N$. Let $M$ be a self-dual Dieudonn\'e lattice in $N$ with $a(M)=1$ and $\pr_{0}(M)_{\tau}=\Lambda_0\cap N_0$. Such a lattice exists by \cite[Section 4]{polpdiv} where it is shown that $\pr_0(M)$ can be chosen to be any given Dieudonn\'e lattice of $a$-invariant 1 in $N_0$. We claim that  $\pr_i(M_{\tau})\subseteq \Lambda_0\cap N_i$ for $i=0,\tfrac{1}{2},1$, and thus that $M_{\tau}\subseteq \Lambda_0$. From the definition of the symplectic pairing on $\breve N$ we obtain that  $\Lambda_0^{\vee}=\Pi^{g-1}\Lambda_0$. By Lemma \ref{lem33}(3), the same holds for $M_{\tau}$. Thus our claim implies the assertion of the proposition.\\

To prove the claim, we first consider the case that $g$ is even. 
For a lattice $L\subset N_0$ let $$L^{\vee}=\{x\in N_1\mid \langle v,x\rangle\in \breve{\mathbb{Z}}_p\text{ for all }v\in L\}.$$ 

For any lattice $L$ in a sub-isocrystal of $\breve N$ let $L^{\tau}=\bigcap_{i\geq 0}\tau^i(L)$ be the largest $\tau$-stable sublattice of $L$. Self-duality of $M$ implies that $M\cap N_1=\pr_0(M)^{\vee}$. Applying powers of $\tau$ yields $(M\cap N_1)^{\tau}\subseteq \tau^i\pr_0(M)^{\vee}$ for all $i$, and thus 
\begin{equation}\label{eqdualpart1}
M^{\tau}\cap N_1=(M\cap N_1)^{\tau}=(\pr_0(M)_{\tau})^{\vee}=\langle e_{g/2+1},\dotsc, e_g\rangle_{\breve{\mathbb{Z}}_p[\Pi]}=\Pi^{g-1}(\Lambda_0\cap N_1).
\end{equation}

By Lemma \ref{lem33}(2) applied to $N_1$, we have $$(M\cap N_1)_{\tau}=\Pi^{-\tfrac{g}{2}+1}(M\cap N_1)^{\tau}=\Pi^{\tfrac{g}{2}}(\Lambda_0\cap N_1).$$

We write $v=v_0+v_1$ with $v_i\in N_i$. Let $\Ann(v_0)=\{\varphi\in\mathcal D\mid \varphi(v_0)=0\}$. Then $M\cap N_1=\Ann(v_0)\cdot v_1$. By \cite[Lemma 2.6]{polpdiv} we have that $\Ann(v_0)$ is a principal ideal in $\mathcal D$ generated by some $\brZ$-linear combination of the $F^aV^b$ with $a+b\geq g/2$ and which is such that the coefficients of $F^{g/2}$ and $V^{g/2}$ are invertible. Altogether we have $$\Pi^{\tfrac{g}{2}}(\Lambda_0\cap N_1)=(M\cap N_1)_{\tau}=(\Ann(v_0)v_1)_{\tau}=\Pi^{\tfrac{g}{2}}\pr_1(M)_\tau,$$ hence $\pr_1(M)_\tau=\Lambda_0\cap N_1$ as claimed.\\

Let now $g$ be odd. To compute $\pr_{1/2}(M_{\tau})$, we use \cite[Prop.~4.3]{polpdiv}. The value of $m$ in that proposition is in our case equal to $\tfrac{g-1}{2}$, and $c=1$. We obtain $$\pr_{1/2}(M_{\tau})^{\vee}=p^{g/2}\pr_{1/2}(M_{\tau}),$$ hence $\pr_{1/2}(M_{\tau})=\Pi^{-\tfrac{g-1}{2}}\langle e_{m}, f_{m}\rangle_{\brZ}=\pr_{1/2}(\Lambda_0)$ is as above. To compute $M_{\tau}\cap N_1$, one proceeds in the same way as for the case that $g$ is even, thus we omit the details.  
\end{proof}

From now on we denote by $\Lambda_0$ the lattice constructed in the proposition. Building on results of Serre, Karemaker--Yu \cite[Lemma 6.16]{KaremakerYu} prove

\begin{lemma}\label{lemmaKY63}
Let $n\geq 1$ and $s\geq 1$ be integers and let $$V_{p,s}=1+\Pi^s{\rm Mat}_n(O_D)\subseteq \GL_n(\mathcal O_D).$$ Then $V_{p,s}$ is torsion free if 
\begin{enumerate}
\item $s\geq 3$ or
\item $p\geq 3$ and $s=2$ or
\item $p\geq 5$ and $s=1$.
\end{enumerate}
\end{lemma}

\begin{rem}\label{rem38}
Let again $M\subset \breve N$ be the Dieudonn\'e lattice corresponding to some $x\in  \mathcal C_{\Lambda_0}(\overline{\mathbb{F}}_p)$. Summarizing the above discussion we have $$\Aut (M,F,\langle\cdot,\cdot\rangle)\subseteq \Aut (M_{\tau},F,\langle\cdot,\cdot\rangle)\cong  \GSp_g(\mathcal O_D)\subseteq \GL_g(\mathcal O_D).$$
 
Thus, Theorem \ref{thmmain} holds for some fixed $p$ and $g$ if generically on $\mathcal C_{\Lambda_0}$, we have 
\begin{equation}\label{eqintrolast}
\Aut(M,F,\langle\cdot,\cdot\rangle)\subseteq \pm 1\cdot (V_{p,s}\cap J(\mathbb Q_p))
\end{equation}
for some $s$ satisfying the assumption of Lemma \ref{lemmaKY63}.

Conversely, Lemma \ref{lem33} implies that $\pm 1\cdot (V_{p,g-1}\cap J(\mathbb Q_p))\subseteq \Aut(M,F,\langle\cdot,\cdot\rangle)$ for every such $M$.
\end{rem}

\begin{lemma}\label{propopen}
Let $Y$ be the subset of $\mathcal C_{\Lambda_0}$ defined by the condition that for any perfect field $k$ over $\mathbb F_p$, a $k$-valued point lies in $Y$ if and only if $\pm 1$ are the only automorphisms of finite order of the associated triple $(M,F,\langle\cdot,\cdot\rangle)$. Then $Y$ is the underlying subset of an open subscheme of $\mathcal C_{\Lambda_0}$.
\end{lemma}

\begin{proof}
By the previous remark, the possible automorphisms of such triples $(M,F,\langle\cdot,\cdot\rangle)$ are elements of $\GSp_g(\mathcal O_D)$. By Lemma \ref{lem33}, $\GSp_g(\mathcal O_D)\cap V_{p,g-1}$ acts trivially on $\mathcal C_{\Lambda_0}$. For every element $\bar j_i$ of the finite quotient group $\GSp_g(\mathcal O_D)/\GSp_g(\mathcal O_D)\cap V_{p,g-1}$ fix a representative $j_i$. By the above, the fixed point locus of $j_i$ on $\mathcal C_{\Lambda_0}$ is a closed subscheme of $\mathcal C_{\Lambda_0}$ which only depends on $\bar j_i$ and not on the choice of representative. Thus $Y$ is the intersection of the open complements of these finitely many fixed point loci, hence open. 
\end{proof}

Thus it remains to show that $Y$ as above is non-empty, or in other words that the fixed point locus of each $j_i$ is a proper closed subscheme of $\mathcal{C}_{\Lambda_0}$. In the remainder of this section we study the condition that a given $M=\mathcal D v$ with $M_{\tau}=\Lambda_0$ is also self-dual.

\begin{lemma}
Let $M\subset N$ be a Dieudonn\'e lattice with $a(M)=1$ and $M_{\tau}=\Lambda_0$. Assume that $M^{\vee}\supseteq M$. Then $M$ is self-dual.
\end{lemma}
\begin{proof}
Let $M$ be as in the lemma and let $v$ be a generator as Dieudonn\'e module. By Lemma \ref{lem33}(1) $(M\cap \Pi^j\Lambda_0)/(M\cap \Pi^{j+1}\Lambda_0)$ is a sub-vector space of the $g$-dimensional $\overline{\mathbb{F}}_p$-vector space $\Pi^j\Lambda_0/\Pi^{j+1}\Lambda_0$ generated by the elements $F^lV^mv$ with $l+m=j$, a $\min\{j+1,g\}$-dimensional sub-vector space. In particular, the volume of $M$ is equal to $\vol\Lambda_0-\sum_{j=0}^{g-2}(g-1-j)=\vol\Lambda_0-\frac{g(g-1)}{2}=\vol \langle e_i,f_i\mid i=1,\dotsc,g\rangle=\vol \langle e_i,f_i\mid i=1,\dotsc,g\rangle^{\vee}=0$. Thus $\vol M=\vol M^{\vee}$, and $M^{\vee}\supseteq M$ implies that the two lattices are equal.
\end{proof}
\begin{rem}\label{reminitcond}
Assume that $M=\mathcal D v$. Then as in \cite[Rem.~4.1]{polpdiv} the condition $M^{\vee}\supseteq M$ is equivalent to $$\langle v, F^jv\rangle,\langle v, V^jv\rangle\in \breve{\mathbb{Z}}_p $$ for all $j\geq 0$. For $j=0$, this pairing is $0$. Furthermore,
$$\langle v, V^jv\rangle=\langle F^jv, v\rangle^{\sigma^{-j}}=- \langle v,F^jv\rangle^{\sigma^{-j}},$$
thus the conditions for the $F^j$ imply those for the $V^j$. Also note that for $j\geq g-1$ we have $F^j v\in \Pi^{g-1}\Lambda_0=\Lambda_0^{\vee}$ which implies that $\langle v, F^jv\rangle  \in \breve{\mathbb{Z}}_p $ for all $v\in \Lambda_0$. 

As in \eqref{eqcoordv} we write
\begin{equation}\label{eqcoordvpi}
v=\sum_{i=1}^g \sum_{l=0}^{\infty} [a_{i,l}]\Pi^l X_i.
\end{equation}

Then 
\begin{align}\label{eqpairingincoord}
\nonumber\langle v,F^j v\rangle=& \sum_{i=1}^g\sum_{l,l'\geq 0} [a_{i,l}a_{g+1-i,l'}^{\sigma^j}]\langle \Pi^l X_i,\Pi^{l'+j} X_{g+1-i}\rangle\\
=& \sum_{i=1}^g\sum_{l,l'\geq 0} (-1)^l[a_{i,l}a_{g+1-i,l'}^{\sigma^j}]\langle X_i,\Pi^{l+l'+j} X_{g+1-i}\rangle.
\end{align}
Let us compute $\langle X_i,\Pi^{j} X_{g+1-i}\rangle$ for all values of $g,i,j$.
\begin{itemize}
\item If $i\leq g/2$, then $X_i=e_i$ and $X_{g+1-i}=\Pi^{-g+1}e_{g+1-i}$. Hence the pairing is $p^{(j-g)/2}$ if $j-g$ is even, and $0$ otherwise. 
\item If $i=m=(g+1)/2$ (and $g$ is odd), the pairing is $$\langle\Pi^{-(g-1)/2} e_m,\Pi^{-(g-1)/2+ j} e_m\rangle=(-1)^{(g-1)/2}p^{-(g-1)/2+(j-1)/2}$$ if the exponent is an integer, and $0$ otherwise.
\item If $i>(g+1)/2$, then  
$$\langle X_i,\Pi^{j} X_{g+1-i}\rangle=\langle\Pi^{-g+1} X_i,\Pi^{j} X_{g+1-i}\rangle=(-1)^{g-1}p^{(j-g)/2} $$ if the exponent is an integer, and $0$ otherwise.
\end{itemize}
\end{rem}

Our next step is to analyse the set of solutions of the resulting system of equations for the $a_{i,l}$. We first consider a system of equations that will occur several times in the proof.

\begin{prop}\label{propSysEq}
Let $n=\lfloor \tfrac{g-1}{2}\rfloor$ and let $a_{1,0},\dotsc, a_{n ,0}\in \overline{\mathbb{F}}_p$ be not contained in any $\mathbb{F}_{p^2}$-rational hyperplane. We consider the system of $n$ equations
$$\sum_{i=1}^n ((-1)^{g-1}a_{i,0}^{\sigma^j}z_i+a_{i,0}z_i^{\sigma^j})-d_j=0$$
 where $j$ runs through all even (resp.~all odd) integers with  $1\leq j\leq g-1$ and where the $d_j$ are fixed elements of $\overline{\mathbb{F}}_p$.
 This system of equations has finitely many solutions. More precisely, the number of solutions is equal to $\prod_j p^j$ where the indexing set is the same as above. 
\end{prop}
\begin{proof}
We consider the map $\Phi:\mathbb{A}^n\rightarrow \mathbb{A}^n$ mapping $(z_1,\dotsc, z_n)$ to the tuple consisting of the left hand sides of the above equations. The Jacobian matrix of this is constant and up to sign equal to the matrix $(a_{i,0}^{\sigma^j})_{i,j}$. By Lemma \ref{lemailinindep}, this matrix is invertible. Thus $\Phi$ is \'etale, in particular, the system of equations has only finitely many solutions. By B\'ezout's theorem, the associated homogenous system of equations (using an additional variable $z$ to homogenize) has the claimed number of solutions, each with multiplicity 1. It remains to show that there is no solution satisfying $z=0$. Assume that we have such a solution $[z_1:\dotsm: z_n:0]$. Then the $z_i$ satisfy $$\sum_i a_{i,0}z_i^{\sigma^j}=0$$ for all $j$ as above. Applying $\sigma^{g-1-j}$ to the $j$th equation, this implies that $$\sum_i a_{i,0}^{\sigma^{g-1-j}}z_i^{\sigma^{g-1}}=0$$ for all $j$. The matrix  $(a_{i,0}^{\sigma^{g-1-j}})_{i,j}$ is invertible by Lemma \ref{lemailinindep}. Hence $z_i^{\sigma^{g-1}}=0$ for all $i$, and thus all $z_i$ vanish, a contradiction.
\end{proof}
\begin{rem}\label{remexplsol}
 If $(z_i)$ and $(z'_i)$ are two solutions of the above system of equations, then their difference is a solution of the associated system of equations with all $d_j=0$. We need an explicit parametrization of the set of solutions of this system with vanishing constant coefficients for the case that $j\equiv g\pmod{2}$. For all $i< l\leq n$ let $b_{i,l}\in \mathbb F_{p^2}$ and for all $i$ let 
$$b_{i,i}\in \begin{cases}
\varepsilon\mathbb{F}_p&\text{for odd }g\\
\mathbb{F}_{p^2}&\text{for even }g
\end{cases}$$
where $\varepsilon\in\mathbb{F}_{p^2}$ with $\varepsilon^{\sigma}=-\varepsilon$. We associate with any such choice an $n$-tuple $(x_i)$ by setting 
$$x_i=b_{ii}a_{i,0}+\sum_{l>i} b_{i,l}a_{l,0}+(-1)^g\sum_{l<i}b_{l,i}^{\sigma^j}a_{l,0}.$$
Since the coordinates are $\mathbb{F}_{p^2}$-linear combinations of the $a_{i,0}$, these $n$-tuples are mutually distinct and give as many points of $\mathbb{A}^n$ as the number of solutions of the system of equations. One easily checks that they satisfy the equations (with vanishing $d_j$) and thus all solutions are of this form.

Phrased more conceptually, this description is part of the assertion that the set of solutions is in one orbit under $J(\mathbb{Z}_p)$.
\end{rem}

\begin{prop}\label{propcondai}
The set 
$$\{v\in\breve N\mid (\mathcal D v)_{\tau}=\Lambda_0, (\mathcal D v)^{\vee}=\mathcal D v\}$$ agrees with the set of elements $v\in \Lambda_0$ obtained in the following way. Writing $v$ uniquely as a linear combination of the basis we obtain $$v=\sum_{i=1}^g\sum_{l\geq 0} [a_{i,l}]\Pi^l X_i$$ with $a_{i,l}\in \overline{\mathbb{F}}_p.$ Then 
\begin{enumerate}
\item The $a_{i,l}$ with $i\leq \lceil \tfrac{g+1}{2}\rceil$ can be chosen arbitrarily under the condition that the $a_{i,0}$ are linearly independent over $\mathbb{F}_{p^2}$.
\item The $a_{i,l}$ for $i> \lceil \tfrac{g+1}{2}\rceil$ and $l\geq g-1$ can also be chosen arbitrarily. 
\item The $a_{i,0}$ for $i> \lceil \tfrac{g+1}{2}\rceil$ are defined as $a_{i,0}:=z_{g+1-i}$ where $(z_1,\dotsc, z_{\lfloor \tfrac{g-1}{2}\rfloor})$ is any solution of the system of equations
\begin{equation}
\sum_{i=1}^{\lfloor \tfrac{g-1}{2}\rfloor}((-1)^{g-1}z_ia_{i,0}^{\sigma^{j}}+ z_i^{\sigma^j}a_{i,0})=d_{j,0}
\end{equation}
for $1\leq j< g-1$ and $j\equiv g\pmod{2}$. Here, 
$$d_{j,0}=\begin{cases}
-a_{g/2,0}a_{g/2+1,0}^{\sigma^j}+a_{g/2,0}^{\sigma^j}a_{g/2+1,0}&\text{if }g\text{ is even,}\\
(-1)^{m}a_{m,0}a_{m,0}^{\sigma^j}&\text{if }g\text{ is odd.}
\end{cases}$$
\item The $a_{i,l}$ for $i> \lceil \tfrac{g+1}{2}\rceil$ and $1\leq l<g-1$ are chosen iteratively for increasing $l$. For some fixed $0\leq l\leq g-2$ and $i>\lceil \tfrac{g+1}{2}\rceil$ let $a_{i,l}:=z_{g+1-i}$ where $(z_1,\dotsc, z_{\lfloor(g-1)/2\rfloor})$ is any solution of the system of equations of the form
\begin{equation}
\sum_{i=1}^{\lfloor \tfrac{g-1}{2}\rfloor}((-1)^{g-1}z_ia_{i,0}^{\sigma^{j}}+ z_i^{\sigma^j}a_{i,0})=d_{j,l}
\end{equation}
for $1\leq j< g-1-l$ and $j\equiv g+l\pmod{2}$. Here, the $j$-th such equation is obtained by reducing the expression \eqref{eqpairingincoord},
which is by our induction on $l$  contained in $p^{(-g+1+j+l)/2}\brZ$, modulo $p^{(-g+1+j+l)/2+1}\brZ$. 
\end{enumerate}
In particular, the linear independence of the $a_{i,0}$ for $i\leq \lceil \tfrac{g+1}{2}\rceil$ implies linear independence over $\mathbb{F}_{p^2}$ of all $a_{i,0}$.
\end{prop}

\begin{rem}\label{remcondaiprop}
\begin{enumerate}
\item Note that by Proposition \ref{propSysEq}, for every choice in (1) and (2), every system of equations of the form in (3) and (4) has solutions. For $l>0$, there will be infinitely many solutions because the number of equations is less than it is in Proposition \ref{propSysEq}.
\item In the following, we only need the explicit description of the $a_{i,l}$ for $l=0$ and the assertion that for every choice of the $a_{i,0},a_{i,1}$ as above, one can indeed find an associated element $v$ generating a self-dual lattice. 
\item Notice that the condition in (3) is a particular case of (4), it is only written separately because in this case, the equation is more explicit and at the same time for the particular case that we use the most. Also, the (non-existence of a) condition for $l\geq g-1$ as in (2) is a particular case of (4). Indeed, for large $l$, there is no $j$ as in (4). Again, this part of the proposition is only to highlight that the conditions are particularly simple in that range. 
\item The main assertion of this proposition is the last sentence, the  others are rather a (very helpful) rearrangement of the conditions that $\langle v,F^jv\rangle\in\brZ$ for all $j$.
\item For $g=4$, this is the analogue of the coordinates studied by Harashita in \cite{HarPpav4}, and used by Karemaker and Yu in \cite[Section 7]{KaremakerYu}.
\end{enumerate}
\end{rem}

\begin{ex}
As an illustration consider the particular case that $g=5$. Then $v\in \Lambda_0$ is of the form $\sum_{i=1}^5 \sum_l[a_{i,l}]\Pi^l X_i$ with $X_1=e_1$, $X_2=e_2$, $X_3=p^{-1}e_3$, $X_4=p^{-2}e_4$ and $X_5=p^{-2}e_5$. The vector $v$ generates a Dieudonn\'e lattice with $(\mathcal D\cdot v)_{\tau}=\Lambda_0$ if and only if the $a_{i,0}$ are linearly independent over $\mathbb{F}_{p^2}$. Self-duality of $\mathcal D\cdot v$ is equivalent to the condition $\langle v, F^jv\rangle\in \breve{\mathbb{Z}}_p$ for $j=1,2,3$. Computing the pairing modulo $\breve{\mathbb{Z}}_p$, one sees that the condition for $j=3$ is equivalent to $$\sum_{i=1}^5 a_{i,0}a_{6-i,0}^{\sigma^3}=0\in \overline{\mathbb{F}}_p,$$ the condition for $l=0$ and $j=3$ of the proposition, with $d_{3,0}=-a_{3,0}a_{3,0}^{\sigma^3}$. 

For $j=2$ we get
$$\sum_{i=1}^5 a_{i,0}a_{6-i,1}^{\sigma^2}=0\in \overline{\mathbb{F}}_p.$$ This is the condition for $l=1$ and $j=2$ in the above proposition, with $d_{2,1}= -a_{3,0}a_{3,1}^{\sigma^2}$.

For $j=1$ we obtain in the same way
$$\sum_{i=1}^5 [a_{i,0}a_{6-i,0}^{\sigma}]+p\left(\sum _{i=1}^5 [a_{i,0}a_{6-i,2}^{\sigma}]+\sum_{i=1}^5 [a_{i,1}a_{6-i,1}^{\sigma}]\right)\in p^2\breve{\mathbb{Z}}_p.$$ Taken modulo $p$, this yields the condition for $l=0$ and $j=1$ of the proposition. We have $d_{1,0}= -a_{3,0}a_{3,0}^{\sigma}$. If this is satisfied, the remaining condition implies the equation for $l=2$ and $j=1$, with $$d_{1,2}\equiv -p^{-1}\sum_{i=1}^5 [a_{i,0}a_{6-i,0}^{\sigma}]-\sum_{i=1}^5 [a_{i,1}a_{6-i,1}^{\sigma}]-[a_{3,0}a_{3,2}^{\sigma}]\pmod{p}.$$
\end{ex}

\begin{proof}[Proof of Proposition \ref{propcondai}]
We have to show that the conditions of the proposition are equivalent to the conditions that all $a_{i,0}$ are linearly independent over $\mathbb{F}_{p^2}$ and that $\langle v,F^jv\rangle\in\brZ$ for all $1\leq j\leq g-2$. By Remark \ref{remcondaiprop}(3), it is enough to consider conditions (1) and (4) of the proposition. We first show that given (1),  the condition that $\langle v,F^jv\rangle\in\brZ$ for all $j$ is equivalent to (4). Recall the decomposition of $\langle v,F^jv\rangle\in p^{\lfloor(-g+1+j)/2\rfloor}\brZ$ given in \eqref{eqpairingincoord}. We iteratively consider the condition $\langle v,F^jv\rangle\in p^{n}\brZ$, for $n=\lfloor(-g+1+j)/2\rfloor+1,\dotsc, 0$. We consider \eqref{eqpairingincoord} modulo $p^n\brZ$ for some such $n$. Among all the summands we consider those containing factors of the form $a_{i,l_0}^{\sigma^c}$ for some $c$ with maximal possible $l_0$. They arise for those summands $$(-1)^l[a_{i,l}a_{g+1-i,l'}^{\sigma^j}]\langle X_i,\Pi^{l+l'+j} X_{g+1-i}\rangle$$ where $\langle X_i,\Pi^{l+l'+j} X_{g+1-i}\rangle$ is of valuation $n-1$, and where $l$ or $l'$ is equal to $0$. From the explicit calculation of these pairings (carried out after \eqref{eqpairingincoord}) one sees that then  $l_0=2n-2-j+g$. The sum in  \eqref{eqpairingincoord} taken modulo $p^n\brZ$ then has the form
\begin{equation}\label{eqtosort}
p^{n-1}\sum_{i=1}^{\lfloor g/2\rfloor} ([a_{i,0}a_{g+1-i,2n-2-j+g}^{\sigma^j}]+(-1)^{g-1} [a_{i,0}^{\sigma^j}a_{g+1-i,2n-2-j+g}])+\Delta
\end{equation}
where $\Delta$ is a sum of expressions of the form $\pm p^c[a_{i,l}a_{g+1-i,l'}^{\sigma^j}]$ for suitable signs and exponents $c<{n-1}$, and such only pairs $(i,l)$ and $(g+1-i,l')$ occur where either the first entry is at most  $\lceil\tfrac{g+1}{2}\rceil$ or the second entry is smaller than $l_0$. From our iterative approach, we may assume that $\Delta\in p^{n-1}\brZ$. The requirement that \eqref{eqtosort} vanishes modulo $p^n$ yields an equation for the $a_{i,2n-2-j+g}$ for all $i>\lceil\tfrac{g+1}{2}\rceil$ in terms of $a_{i,l'}$ with $i\leq \lceil\tfrac{g+1}{2}\rceil$ or with $l'<2n-2-j+g$. Sorting these equations according to the index $l_0=2n-2-j+g$ instead of according to $j$, we obtain for a fixed $l_0$ the system of equations listed in (4). By Proposition \ref{propSysEq}, using condition (1), this system of equations has solutions $a_{i,l}$ for any given values of $a_{i,l'}$ (with $l'<l$) satisfying (1) and (4). Thus the two sets of conditions are equivalent.

It remains to show that for any choice of $a_{i,0}$ satisfying (1) and (3), the $a_{i,0}$ for $1\leq i\leq g$ are linearly independent over $\mathbb{F}_{p^2}$. To ease the notation, we skip all second indices $0$ for the rest of this proof. Assume that the assertion does not hold, i.e.~there is a tuple $(a_{i})$ satisfying (1) and (3), but such that some $a_{i_0}$ is an $\mathbb{F}_{p^2}$-linear combination of the others. Then $i_0> (g+1)/2$. We exchange $a_{i_0}$ and $a_{g}$ as well as $a_{g+1-i_0}$ and $a_{1}$. We obtain an element $v'$ of $\Lambda_0$ satisfying (1) and (3) (the summands in the equations are just permuted) and such that our above assumption holds for $i_0=g$. We may assume that this already holds for $v$. We write $a_{g}=\sum_{1\leq i<g}\delta_ia_{i}$ with $\delta_i\in\mathbb{F}_{p^2}$.\\

\noindent{\it Case 1: $g$ is even.}

In the equations in (3), the summands involving the $a_l$ with indices $1,i,g+1-i,g$ for some $1<i<g$ read
$$a_1a_g^{\sigma^j}-a_ga_1^{\sigma^j}\pm a_ia_{g+1-i}^{\sigma^j}-\pm a_{g+1-i}a_i^{\sigma^j}$$ where the $\pm$-signs are both times $1$ if $i\leq g/2$ and $-1$ otherwise. Thus we see that the validity of the equations in (3) does not change if we replace $a_{g}$ by $a_{g}-\delta_i a_{i}$ and $a_{g+1-i}$ by $a_{g+1-i}\pm\delta_i a_1$. Making such replacements for all $i>1$, we may assume that $a_{g}$ is an $\mathbb{F}_{p^2}$-multiple of $a_{1}$. After this transformation, (1) is also still valid. But then, Remark \ref{remexplsol} shows that we may replace $a_{g}$ by $0$ and obtain again a solution of the equations in (3). Altogether, we obtain a new tuple $(a_i)$ satisfying the relations in (1) and (3), and in addition $a_{g}=0$.

The system of equations then reads 
$$\sum_{i=2}^{g/2}(a_{i}a_{g+1-i}^{\sigma^j}-a_{g+1-i}a_i^{\sigma^j})=0$$ for all even $2\leq j\leq g-2$. Although we assumed $a_1,\dotsc, a_{g/2}, a_{g/2+1}$ to be fixed and linearly independent over $\mathbb{F}_{p^2}$, we now view $a_{g/2+1}$ for a moment as variable. Then this system of $g/2-1$ equations (in the $g/2-1$ variables $a_{g/2+1},\dotsc, a_{g-1}$) is as in Remark \ref{remexplsol}, with all $d_j=0$. The explicit description of the solutions in the remark shows that $a_{g/2+1}$ is an $\mathbb{F}_{p^2}$-linear combination of $a_1,\dotsc,a_{g/2}$, contradicting our choice of $a_{g/2+1}$ in (1). This finishes the proof of the proposition for even $g$.\\

\noindent{\it Case 2: $g$ is odd.}

 An analogous argument as for $g$ even shows that for $i\neq m, 1$ we can replace $a_{g}$ by $a_g-\delta_ia_{i}$ and at the same time $a_{g+1-i}$ by $a_{g+1-i}+\sigma(\delta_i)a_1$ and may thus assume that the corresponding $\delta_i$ vanish. In other words, $a_{g}=\delta_1 a_{1}+\delta_{m}a_{m}$ for certain $\delta_i\in\mathbb{F}_{p^2}$. Inserting this into the system of equations, we get
\begin{align*}
0&= \sum_{i=2}^{m-1}(a_{i}a_{g+1-i}^{\sigma^j}+a_{g+1-i}a_i^{\sigma^j})+(a_{1}a_{g}^{\sigma^j}+a_{1}^{\sigma^j}a_{g})+a_ma_m^{\sigma^j}\\
&=\sum_{i=2}^{m-1}(a_{i}a_{g+1-i}^{\sigma^j}+a_{g+1-i}a_i^{\sigma^j})+(\delta_1^{\sigma}a_{1}a_{1}^{\sigma^j}+\delta_m^{\sigma}a_1a_m^{\sigma^j}+\delta_1a_1a_{1}^{\sigma^j}+\delta_ma_1^{\sigma^j}a_m)+a_ma_m^{\sigma^j}\\
&=\sum_{i=2}^{m-1}(a_{i}a_{g+1-i}^{\sigma^j}+a_{g+1-i}a_i^{\sigma^j})+
(a_m+\delta_m^{\sigma}a_1)(a_m+\delta_m^{\sigma}a_1)^{\sigma^j}+(\delta_1^{\sigma}+\delta_1-\delta_m\delta_m^{\sigma})a_1a_1^{\sigma^j}
\end{align*}
for all odd $j=1,\dotsc, g-2$. Notice that $\delta_1^{\sigma}+\delta_1-\delta_m\delta_m^{\sigma}\in \mathbb{F}_p$.  Let $\varepsilon\in\mathbb{F}_{p^2}$ with $\varepsilon\varepsilon^{\sigma}= \delta_1^{\sigma}+\delta_1-\delta_m\delta_m^{\sigma}$. Put $\tilde a_1=\varepsilon a_1$ and $\tilde a_m=a_m+\delta_m^{\sigma}a_1$. If $\varepsilon\neq 0$, then $\tilde a_1,a_2,\dotsc, a_{m-1},\tilde a_m$ are $\mathbb{F}_{p^2}$-linearly independent and there are $a_{m+1},\dotsc a_{g-1}$ such that for all odd $j=1,\dotsc, g-2$ we have
$$\sum_{i=2}^{m-1}(a_{i}a_{g+1-i}^{\sigma^j}+a_{g+1-i}a_i^{\sigma^j})+
\tilde a_m \tilde a_m^{\sigma^j}+\tilde a_1\tilde a_1^{\sigma^j}=0.$$
If $\varepsilon=0$, then $a_1, a_2,\dotsc, a_{m-1},\tilde a_m$ are $\mathbb{F}_{p^2}$-linearly independent and there are $a_{m+1},\dotsc a_{g-1}$ such that for all odd $j=1,\dotsc, g-2$ we have
$$\sum_{i=2}^{m-1}(a_{i}a_{g+1-i}^{\sigma^j}+a_{g+1-i}a_i^{\sigma^j})+
\tilde a_m \tilde a_m^{\sigma^j}=0.$$
In both cases, we view these as $(g-1)/2=m-1$ equations in the $m-2$ variables $a_{m+1},\dotsc, a_{g-1}$.

On $\overline{\mathbb{F}}_{p}^{g-1}$ we define a symmetric non-degenerate pairing $\left\{\cdot,\cdot\right\}$ by
$$\left\{(x_i)_i,(y_j)_j\right\}=\sum_{i\neq 1,m}x_iy_{g+1-i}+x_my_m+x_1y_1.$$
It also satisfies $$\left\{(\sigma(x_i))_i,(\sigma(y_j))_j\right\}=\sigma(\left\{(x_i)_i,(y_j)_j\right\}).$$ 
Let $\bar v=(\tilde a_1,a_2,\dotsc, a_{m-1},\tilde a_m, a_{m+1},\dotsc, a_{g-1})$. Then the above system of equations is equivalent to 
\begin{equation}\label{eqsyseqpair}
0=\left\{\bar v,\sigma^j(\bar v)\right\}=\left\{\bar v,\sigma^{-j}(\bar v)\right\}
\end{equation} for all odd $j$ in the above range and where $\sigma$ acts componentwise.

We iteratively define for $i=2,\dotsc, m-1$ functions $\Phi_i:\overline{\mathbb{F}}_p\rightarrow \overline{\mathbb{F}}_p$ whose kernels are the $\mathbb{F}_{p^2}$-sub-vector spaces generated by $a_2,\dotsc,a_i$. Let $\Phi_2(x)=x-\frac{a_2}{a_2^{\sigma^2}}\sigma^2(x)$. Let $\tilde a_i=\Phi_{i-1}(a_i)$ and $\Phi_i=\Delta_i\circ \Phi_{i-1}$ with $\Delta_i(x)=x-\frac{\tilde a_i}{\tilde{a}_i^{\sigma^2}}\sigma^2(x)$. Then iteratively we see that $\Phi_i$ has the required kernel, and since we assumed $a_2,\dotsc, a_{m-1}$ to be linearly independent, again $\tilde{a}_i$ is different from zero, so that the $\Phi_i$ are well-defined. By definition, $\Phi_{m-1}$ is an $\overline{\mathbb{F}}_p$-linear combination of the even powers of Frobenius $1, \sigma^2,\dotsc, \sigma^{2(m-2)}$. We denote the induced map $\overline{\mathbb{F}}_{p}^{g-1}\rightarrow\overline{\mathbb{F}}_{p}^{g-1}$ also by $\Phi_{m-1}$. Then the components of $\Phi_{m-1}(\bar v)$ with index $2,\dotsc, m-1$ vanish, but the $m$th component $b_m:=\Phi_{m-1}(\tilde a_{m})$ is nonzero. The first component $b_1:=\Phi_{m-1}(\tilde a_{1})$ is nonzero if and only if $\tilde a_1\neq 0$, i.e. $\varepsilon\neq 0$. From this we see that 
\begin{equation}\label{eqcontr1}
\left\{\Phi_{m-1}(\bar v),\sigma\circ \Phi_{m-1}(\bar v)\right\}=b_1b_1^{\sigma}+b_mb_m^{\sigma}.
\end{equation} 
On the other hand, the left hand side is a linear combination of expressions of the form 
$$\left\{\sigma^{l}(\bar v),\sigma^{l'}(\bar v)\right\}$$ with $l$ even and $l'$ odd, and both between $0$ and $2(m-2)+1$. Since the pairing is symmetric and compatible with $\sigma$, \eqref{eqsyseqpair} implies that each of these expressions is zero. Altogether, the left hand side of \eqref{eqcontr1} is zero, thus 
\begin{equation}\label{eqlast2}
b_1b_1^{\sigma}+b_mb_m^{\sigma}=0.
\end{equation}
In case that $\varepsilon=b_1=0$, this implies that $b_m=0$, i.e. $\tilde a_m$ is in the kernel of $\Phi_{m-1}$. Thus $\tilde a_m$ is linearly dependent on $a_2,\dotsc, a_{m-1}$, a contradiction. If $b_1\neq 0$, \eqref{eqlast2} implies that $$\frac{b_m}{b_1}\cdot \sigma(\frac{b_m}{b_1})=-1,$$ thus $\frac{b_m}{b_1}\in \mathbb{F}_{p^2}$, again a contradiction to $\tilde a_1,a_2,\dotsc,a_{m-1},\tilde a_m$ being linearly independent. Thus also for $g$ odd, our assumption that the $a_i$ (for $1\leq i\leq g$) are linearly dependent over $\mathbb{F}_{p^2}$ is false, which finishes the proof.
\end{proof}

\section{Proof of Theorem \ref{thmmain}}\label{secpfoort}

We consider the $\tau$-stable lattice $\Lambda_0$ and the associated irreducible component $\mathcal C_{\Lambda_0}$ of $\mathcal M_{g}^{\circ}$ studied in the previous section.

\begin{claim}\label{keyclaim}
Let $g\geq 5$. For $s=2$ and all $p$, and for $s=3$ if $p=2$, there is a dense open subscheme $Y_s$ of $\mathcal C_{\Lambda_0}$ such that for every $M\subset N$ corresponding to an $\overline{\mathbb F}_p$-valued point of $Y_s$ and every $j\in J(\mathbb Q_p)$ of finite order stabilizing $\Lambda_0$ and $M$, we have that $\pm j$ induces the identity of $\Lambda_0/\Pi^s\Lambda_0$.
\end{claim}

By Lemma \ref{lemmaKY63} and Remark \ref{rem38}, the claim for $s=1$ implies Oort's conjecture for all $p\geq 5$. Likewise, the claim for $s=2$ implies Oort's conjecture for all $p\geq 3$. For $s=3$, we will only prove the claim if $p=2$. It then implies Oort's conjecture for $p=2$ (and $g\geq 5)$.

\begin{rem}\label{remclaim}
\begin{enumerate}

\item From the proof of Lemma \ref{propopen}, we see that there are only finitely many $j\in J(\mathbb Q_p)$ of finite order stabilizing $\Lambda_0$. Thus it is enough to show that for any fixed $j$ not inducing  $\pm 1$ on $\Lambda_0/\Pi^s\Lambda_0$, the fixed point locus of $j$ in $\mathcal C_{\Lambda_0}$ is not all of $\mathcal C_{\Lambda_0}$. 
\item In view of Proposition \ref{propcondai}, it is thus enough to find for every fixed $j$ not inducing  $\pm 1$ on $\Lambda_0/\Pi^s\Lambda_0$ a $v$ as in that proposition such that $j$ does not stabilize $\mathcal D v$. 
\item To show (1), we proceed iteratively and consider the analogous assertion for $s=0,1,2,$ and $3$, the latter only for $p=2$. For $s=0$, the claim is obviously true. Thus for $s=1,2,3$ we may assume in addition that the fixed element $j$ induces $\pm$id  on $\Lambda_0/\Pi^{s-1}\Lambda_0$.
\end{enumerate} 
\end{rem}

\subsection{Claim \ref{keyclaim} holds for $s=1$.}\label{sec41}

Let $j$ be as in Remark \ref{remclaim}(1), and assume that $j$ induces an automorphism of every $M$ associated with a point of $\mathcal C_{\Lambda_0}$. Then $j$ stabilizes $\Lambda_0$ and hence also $\Pi\Lambda_0=F\Lambda_0$. It thus induces an automorphism $j_1$ of $\Lambda_0/\Pi\Lambda_0$, which is a $g$-dimensional $\overline{\mathbb{F}}_p$-vector space generated by the images of the $X_i$. We denote this basis again by the same letters $X_i$.

{\it Step 1: $j_1$ is a scalar multiple of the identity.}

We consider some $v$ as in Proposition \ref{propcondai} and let $M$ be the associated Dieudonn\'e lattice. The automorphism $j_1$ stabilizes the image of $M+\Pi\Lambda_0/\Pi\Lambda_0$, a 1-dimensional sub-vector space of $\Lambda_0/\Pi\Lambda_0$ generated by the image $\bar v$ of $v$. In particular, $\bar v$ is an eigenvector of $j_1$. Let $\lambda$ be the associated eigenvalue. Since $j$ commutes with $F$ and with multiplication by $p$, we obtain that $j_1$ commutes with $\tau$. Hence for all $i$, also $\tau^i \bar v$ is an eigenvector of $j_1$, with eigenvalue $\sigma^{2i}(\lambda)$. Since $M_{\tau}=\Lambda_0$, we obtain that $\Lambda_0/\Pi\Lambda_0$ is generated by the $\tau^i(\bar v)$ for $i=0,\dotsc, g-1$. In particular, $j_1$ is semisimple. We claim that there is some $v$ such that the vector $\tau^{g}(\bar v)$ (which is also an eigenvector of $j_1$) cannot be expressed as a linear combination of $\lfloor g/2\rfloor$ or less of these $g$ basis vectors. Indeed, the first $\lfloor \tfrac{g}{2}\rfloor+1$ of the coordinates $a_{i,0}$ of $\bar v$ can be chosen freely by Proposition \ref{propcondai}. Consider the matrix consisting of the $a_{i,0}^{\sigma^{2j}}$ for $j=0,\dotsc, g$ and $i\leq \lfloor \tfrac{g}{2}\rfloor+1$. All minors of maximal size of this matrix are polynomials in the $a_{i,0}$ for $i\leq \lfloor \tfrac{g}{2}\rfloor+1$. Considering the multidegrees of the involved monomials, one sees that these polynomials are not identically zero. Hence, there is a $v$ such that none of them vanishes at $\bar v$. On the other hand, the condition that the vector $\tau^{g}(\bar v)$ cannot be expressed as a linear combination of $\lfloor g/2\rfloor$ of the other vectors is equivalent to the condition that all minors of maximal size of the matrix and involving the last column are non-zero. This proves our claim.

Since all of the basis vectors as well as $\tau^g(\bar v)$ are eigenvectors for $j_1$, and since $\tau^g(\bar v)$ is a linear combination of the basis vectors $\tau^j(\bar v)$ for $j<g$, at least $\lfloor g/2\rfloor+1$ of them, as well as $\tau^g(\bar v)$, have to be eigenvectors for the same eigenvalue $\sigma^{2g}(\lambda)$. Hence $\sigma^{2i}(\lambda)=\sigma^{2i+2}(\lambda)$ for some $i\leq g$ which implies that $\lambda$ is fixed by $\sigma^2$. Thus all eigenvalues agree and $j_1$ is a scalar.\\

\noindent{\it Step 2: $\sigma(j_1)= j_1$}

The quotient $\Lambda_0/p\Lambda_0$ is a $2g$-dimensional $\overline{\mathbb F}_p$-vector space with basis consisting of the classes of the $X_i$ and $Y_i$, which we denote again by the same letters. The image $M_2$ of $M$ has as basis the images of $v$, $Fv$ and $Vv$. We denote the image of $v$ by $v_2$ and let $j_2$ be the automorphism of $\Lambda_0/p\Lambda_0$ induced by $j$. Then $$v_2=\sum_{i=1}^{g}([a_{i,0}]X_i+[a_{i,1}]Y_i).$$ 

Let $\lambda$ be as in Step 1. Since $j_2$ commutes with $F$, it multiplies each $Y_i=F(X_i)\in (M+p\Lambda_0)/p\Lambda_0$ with $\sigma(\lambda)$. Furthermore, each $X_i$ is mapped to $\lambda X_i+\sum_{l}\delta_{li}Y_l$ for some coefficients $\delta_{li}\in\mathbb{F}_{p^2}$. Thus $\lambda v_2-j_2v_2\in \Pi \Lambda_0/p\Lambda_0$ satisfies 
\begin{equation}\label{eqstep2}
\lambda v_2-j_2v_2=\sum_{i=1}^{g}\bigl((\lambda-\sigma(\lambda))a_{i,1}-\sum_{l=1}^g \delta_{il}a_{l,0}\bigr)Y_i.
\end{equation}
Notice that $\sum_{l=1}^g \delta_{il}a_{l,0}$ only depends on ($j_2$ and) the $a_{l,0}$ whereas the first summand $(\lambda-\sigma(\lambda))[a_{i,1}]$ depends only on $j_2$ and the coefficients $[a_{i,1}]$. If $j$ is an automorphism of $M$, then \eqref{eqstep2} is contained in the sub-vector space $M\cap \Pi\Lambda_0/M\cap p\Lambda_0\subseteq \Pi \Lambda_0/p\Lambda_0$. By Lemma \ref{lem33} it is freely generated by the images of $F(v)$ and $V(v)$, two vectors only depending on the coefficients $a_{i,0}$. Recall that the $a_{i,1}$ for $i\leq \lfloor \tfrac{g}{2}\rfloor +1$ can be chosen freely, and $\lfloor \tfrac{g}{2}\rfloor +1\geq 3$ since $g>3$. Thus for generic $a_{i,1}$, this condition can only be satisfied if $\lambda-\sigma(\lambda)=0$, which finishes Step 2.\\

Finally we use  that $j\in \Sp_{g}(\mathcal O_D)$, hence $\sigma(\lambda)\cdot \lambda=1$. Since $j_1$ is invariant under $\sigma$, we obtain that ${j_1}^2=1$, or $j_1=\pm 1$.

\subsection{Claim \ref{keyclaim} holds for $s=2$}

Let again $j_2$ be the automorphism of $\Lambda_0/p\Lambda_0$ induced by $j$. Possibly multiplying by $-1$, we may assume that $j_1=1$.

 Our assumption on $j_1$ implies that $j_2$ satisfies $j_2(X_i)=X_i+\sum_j \delta_{ji}Y_j$ with $\delta_{ji}\in \mathbb F_{p^2}$ and $j_2(Y_i)=Y_i$ for all $i$. We want to show that all $\delta_{ji}$ are equal to $0$. We have
$$j_2(v)-v=\sum_j\sum_i a_{i,0}\delta_{ji}Y_j\in \Pi\Lambda_0/p\Lambda_0.$$ Thus this vector is in $M_2$ if and only if it lies in the span of the images of $Fv$ and $Vv$ in this quotient, which are $\sum_i \sigma(a_{i,0})Y_i$ and $\sigma^{-1}(a_{i,0})Y_i$. Consider the $3\times g$-matrix whose rows are the images under $\sigma$ of these three vectors,
\begin{equation}\label{eqclaims2}
 \left( \begin{array}{ccc} 
\sigma^2(a_{1,0})&\dotsm&\sigma^2(a_{g,0})\\
a_{1,0}&\dotsm&a_{g,0}\\
\sigma(\sum_i \delta_{1i}a_{i,0})&\dotsm&\sigma(\sum_i \delta_{gi}a_{i,0})
\end{array}\right).
\end{equation}
Then $j_2$ maps $M_2$ to $M_2$ if and only if this matrix has rank 2, i.e.~all of its $3\times 3$-minors vanish.

Assume that there is a $\delta_{l_0i}\neq 0$. Then we want to show that generically, the above matrix has rank 3. If $l_0>\lfloor\tfrac{g}{2}\rfloor+1$, then we can make a base change exchanging $X_{l_0}$ and $X_{g+1-l_0}$ and analogously $Y_{l_0}$ and $Y_{g+1-l_0}$ and hence assume that $l_0\leq \lfloor\tfrac{g}{2}\rfloor+1$. Consider the minor of \eqref{eqclaims2} corresponding to the three columns $i_1=l_0,i_2,i_3$ with all $i_l\leq \lfloor\tfrac{g}{2}\rfloor+1$. 

First consider the case that for all $n>\lfloor\tfrac{g}{2}\rfloor+1$, we have $\delta_{i_ln}=0$ for $l=1,2,3$. Then the above minor is a non-zero polynomial in $a_{1,0},\dotsc, a_{\lfloor\tfrac{g}{2}\rfloor+1,0}$, variables which vary in an open subset of $\mathbb{A}^{\lfloor\tfrac{g}{2}\rfloor+1}$. Hence generically, this determinant is nonzero, and $j_2$ does not stabilize $M_2$. Consider now the remaining case that there is an $n>\lfloor\tfrac{g}{2}\rfloor+1$ with $\delta_{i_ln}\neq 0$ for some $l\leq 3$. We fix values for all $a_{i,0}$ that satisfy the requirements of Proposition \ref{propcondai}. Then the system of equations in Proposition \ref{propSysEq} has at least $p-1$ further solutions for the $a_{i,0}$ that have the same values for all $i\neq n$, but differ in the value of $a_{n,0}$. 

On the other hand, view 
\begin{equation*}
\det \left( \begin{array}{ccc} 
\sigma^2(a_{i_1,0})&\sigma^2(a_{i_2,0})&\sigma^2(a_{i_3,0})\\
a_{i_1,0}&a_{i_2,0}&a_{i_3,0}\\
\sigma(\sum_i \delta_{i_1i}a_{i,0})&\sigma(\sum_i \delta_{i_2i}a_{i,0})&\sigma(\sum_i \delta_{i_3i}a_{i,0})
\end{array}\right)
\end{equation*}
as a polynomial in $a_{n,0}$ with coefficients determined by the $\delta_{li}$ and the remaining fixed $a_{i,0}$. Since $a_{n,0}$ only occurs in the last row, this polynomial is of degree at most $p$ and purely inseparable. Thus it is zero for a single value of $a_{n,0}$. Choosing any of the other at least $p-1$ possible values for $a_{n,0}$ shows that there is some $v$ such that the minor is non-zero. Thus the Dieudonn\'e lattice $M$ generated by $v$ is not stabilized by $j$. This finishes the proof of Claim \ref{keyclaim} for $s=2$.

\subsection{Claim \ref{keyclaim} holds for $s=3$ and $p=2$}\label{sec43}

We proceed in a similar way as in the preceeding subsection. We may assume that $j$ is such that $j_2=1$. Let $j_3$ be the automorphism of $\Lambda_0/p\Pi\Lambda_0$ induced by $j$. Let $M_3$ be the image of $M$ in $\Lambda_0/p\Pi\Lambda_0$. Again we denote the images of the $X_i$ and $Y_i$ in $\Lambda_0/p\Pi\Lambda_0$ by the respective same letters.\\

It is enough to show that $j_3$ is a scalar multiple of the identity. Indeed, then $j_3=(1+[\lambda] 2)\cdot \id$ for some scalar $1+[\lambda] 2\in \mathbb Z_4/(4)$. Compatibility with the polarization implies that $j_3\sigma(j_3)=1\in \mathbb Z_4/(4)$, thus $\lambda=\sigma(\lambda)$. Hence $\lambda\in \mathbb{F}_2$, and $j_3=\pm 1\in \mathbb{Z}_2/(4)=\mathbb Z/4\mathbb Z$. This then implies Claim \ref{keyclaim}.

We want to show that as soon as $j_3$ is not a scalar, there is some $v$ as in Proposition \ref{propcondai} such that $j_3(v)-v\notin M_3$. 
Since $j_2=1$ we have  $j_3(Y_i)=Y_i$ for all $i$ and $$j_3(X_i)=X_i+\sum_j \delta_{ji}pX_j$$ for certain $\delta_{ji}\in \mathbb F_{p^2}$. Thus for any $v$ as in the proposition,
$$j_3(v)-v=\sum_j\sum_i a_{i,0}\delta_{ji}pX_j\in p\Lambda_0/\Pi p\Lambda_0.$$
The quotient $p\Lambda_0/\Pi p\Lambda_0$ is a $g$-dimensional $\overline{\mathbb{F}}_p$-vector space generated by the (images of the) $pX_i$. By Lemma \ref{lem33}, the image of $M\cap p\Lambda_0$ is a sub-vector space generated by $F^2v, V^2 v$ and $pv$, which are also linearly independent. In coordinates, they are given by 
$$pv=\sum_i a_{i,0}pX_i,\quad F^2v=\sum_i a_{i,0}^{\sigma^2}pX_i, \quad V^2v=\sum_i a_{i,0}^{\sigma^{-2}}pX_i.$$

From these equations we can read off the coordinates of the images of the four vectors $V^2v,pv,F^2 v, j_3(v)-v$ in $p\Lambda_0/\Pi p\Lambda_0$ with respect to the basis consisting of the images of the $pX_i$. We consider the $4\times g$-matrix $A$ with rows equal to the images under $\sigma^2$ of these coordinates. In other words, the $i$th column of $A$ is
$$(a_{i,0},a_{i,0}^{\sigma^2},a_{i,0}^{\sigma^4},\sum_j \delta_{i,j}a_{j,0}^{\sigma^2} )^T.$$ We have to show that if this matrix has rank 3 on all of $\mathcal C_{\Lambda_0}$, then $j_3$ is a scalar.\\

Before we begin with the actual proof, we provide a lemma.
\begin{lemma}\label{lempreps3}
Let $B$ be a $4\times 4$-matrix with columns $$(x_i,x_i^{\sigma^2},x_i^{\sigma^4},\varepsilon_i )^T$$ for certain fixed $\varepsilon_i\in\mathbb F_{p^2}$ and $\mathbb{F}_{p^2}$-linearly independent $x_1,\dotsc,x_4\in\overline{\mathbb{F}}_p$. Then $\det B=0$ if and only if all $\varepsilon_i$ are zero.
\end{lemma}
\begin{proof}
Assume that some $\varepsilon_i\neq 0$, without loss of generality for $i=1$. For every $j\neq 1$ we subtract $\varepsilon_j/\varepsilon_1$ times the first column from the $j$th. Let $y_1=x_1$, and $y_j=x_j-\tfrac{\varepsilon_j}{\varepsilon_1} x_1 $. Then the matrix still has the same determinant, but reads
$$\left( \begin{array}{cccc} 
y_1&y_2&y_3&y_4\\
y_1^{\sigma^2}&y_2^{\sigma^2}&y_3^{\sigma^2}&y_4^{\sigma^2}\\
y_1^{\sigma^4}&y_2^{\sigma^4}&y_3^{\sigma^4}&y_4^{\sigma^4}\\
\varepsilon_1&0&0&0
\end{array}\right).$$
Since the $y_2,y_3,y_4$ are linearly independent over $\mathbb{F}_{p^2}$, the determinant is nonzero by Lemma \ref{lemailinindep}, as claimed.
\end{proof}

We claim that it is enough to show that $j_3\in \GSp_{g}(\mathbb{Z}_{p^2})$ is a diagonal matrix. Indeed, this then holds as well after any base change in $\GSp_g(\mathbb{Z}_{p^2})\subseteq\GSp_g(\mathcal O_D)$. Hence $j_3$ is contained in every $\mathbb{Q}_{p^2}$-rational torus of $\GSp_g$ and is thus central. To show that all $\delta_{il}=0$ if $l\neq i$, we distinguish two cases.

\noindent{\it Case 1: $g\geq 5$ is odd and $l=m=(g+1)/2$.}

Let $S$ be any $4\times 4$-submatrix of $A$ obtained by selecting some columns $i_1,i_2,i_3,i_4$, all different from $l$. We want to show that $\delta_{i_jm}=0$ for all $j$. Fix some solution $(a_{i,0})_i$ of the equations in Proposition \ref{propcondai}(1) and (3). Let $\varepsilon\in\mathbb{F}_{p^2}\setminus \mathbb{F}_p$. Then $\varepsilon^{\sigma} \varepsilon=\varepsilon^3=1$, hence replacing $a_{m,0}$ by $\varepsilon a_{m,0}$ and keeping all other coordinates yields a different solution $(a_{i,0}')_i$ of the equations in Proposition \ref{propcondai}. Let $S$ and $S'$ be the associated $4\times 4$-matrices, both assumed to be of rank 3. The condition that the ranks of these matrices are 3 is equivalent to $\det S=\det S'=0$. We use that determinants are linear in the last row, and that $S$ and $S'$ agree except for their last rows. Taking the difference, we obtain that 
$$a_m^{\sigma^2}\cdot\det\left( \begin{array}{cccc} 
a_{i_1}&a_{i_2}&a_{i_3}&a_{i_4}\\
a_{i_1}^{\sigma^2}&a_{i_2}^{\sigma^2}&a_{i_3}^{\sigma^2}&a_{i_4}^{\sigma^2}\\
a_{i_1}^{\sigma^4}&a_{i_2}^{\sigma^4}&a_{i_3}^{\sigma^4}&a_{i_4}^{\sigma^4}\\
(\varepsilon-1)\delta_{i_1m}&(\varepsilon-1)\delta_{i_2m}&(\varepsilon-1)\delta_{i_3m}&(\varepsilon-1)\delta_{i_4m}
\end{array}\right)=0.$$
Here, $a_{m,0}^{\sigma^2}$ times the last row is the difference of the last rows of $S$ and $S'$. Thus by Lemma \ref{lempreps3}, all $\delta_{i_j m}$ are zero. Since the $i_j$ were chosen arbitrarily among the indices different from $m$, we obtain $\delta_{i,m}=0$ for all $i\neq m$.\\

\noindent{\it Case 2: $g\geq 5$ and $l\neq g+1-l$.}

We proceed in a similar way: Let $(a_{i,0})_i$ be as in Proposition \ref{propcondai}. This time we can replace the value of $a_{l,0}$ by $a_{l,0}+a_{g+1-l,0}$ and keep all other $a_{i,0}$ fixed to obtain a new solution of the equations of Proposition \ref{propcondai}, compare Remark \ref{remexplsol}. Let $A'$ be the corresponding analogue of $A$. Let $S$ resp. $S'$ be the $4\times 4$-submatrices of $A$ resp. $A'$ obtained by selecting some columns $i_1,i_2,i_3,i_4$, all different from $l$. Our assumption that $A$ and $A'$ have rank at most 3 implies that $\det S=\det S'=0$. Again, $S$ and $S'$ have the same first to third rows, and linearity of the determinant in the last row yields
 $$a_{g+1-l}^{\sigma^2}\det\left( \begin{array}{cccc} 
a_{i_1}&a_{i_2}&a_{i_3}&a_{i_4}\\
a_{i_1}^{\sigma^2}&a_{i_2}^{\sigma^2}&a_{i_3}^{\sigma^2}&a_{i_4}^{\sigma^2}\\
a_{i_1}^{\sigma^4}&a_{i_2}^{\sigma^4}&a_{i_3}^{\sigma^4}&a_{i_4}^{\sigma^4}\\
\delta_{i_1l}&\delta_{i_2l}&\delta_{i_3l}&\delta_{i_4l}
\end{array}\right)=0.$$
Again, by Lemma \ref{lempreps3}, all $\delta_{i_j l}$ are zero, and hence all $\delta_{il}$ for $i\neq l$ are zero.\qed

\section{The non-polarized case}\label{secnonpol}

In this last section we explain how the above arguments also lead to an analogous result on automorphisms of generic (non-polarized) supersingular $p$-divisible groups of dimension $g\geq 5$. Further generalizations to the other non-polarized cases (for smaller $g$, but in particular also for basic $p$-divisible groups that are no longer supersingular) are shown in joint work of the author with Paul Philippe and Fabian Schnelle \cite{oortgln}.

Let $\mathbb{X}$ be as before a supersingular $p$-divisible group of dimension $g\geq 5$. It is unique up to isogeny, and we assume that $\mathbb{X}$ is as in the preceeding sections. Consider as in \cite[2]{RZ} the formal scheme representing the functor assigning to any scheme $S$ such that $p$ is locally nilpotent on $S$ the set 
$$\mathcal M^{{\rm np}}_g(S):=\{(X,\rho)\}/\sim$$ where $X$ is a $p$-divisible group over $S$, and where $\rho: \mathbb{X}_{\bar S}\rightarrow X_{\bar S}$ is a quasi-isogeny over the reduction modulo $p$ of $S$. We denote by $\mathcal M^{{\rm np}}_g$ the underlying reduced subscheme of this moduli space. Then by \cite{RZ}, the forgetful map induces a closed embedding $$\mathcal M_g\rightarrow \mathcal M_g^{{\rm np}}.$$ By \cite[Lemma 4.7]{modpdiv}, the locus $\mathcal M_g^{{\rm np},\circ}$ where the $a$-invariant of the universal $p$-divisible group is $1$ is again open and dense. The main goal of this section is the following analogue of Theorem \ref{thmgenautRZ}.

\begin{thm}\label{thmgenauRZnp}
There is an open and dense subscheme $Y\subseteq \mathcal M^{{\rm np}}_g$ such that for any $y\in Y$, the reduction $X_y$ of the universal $p$-divisible group $X$ on $\mathcal M^{{\rm np}}_g$ does not have any automorphisms of finite order other than unit-root scalars in $\mathbb{Z}_p^{\times}$. 
\end{thm}

\begin{rem}\label{remscalarZp}
Scalars in $\mathbb{Z}_p^{\times}$ are clearly contained in the automorphism group of every $p$-divisible group. 

If $p>2$, the unit roots in $\mathbb{Z}_p^{\times}$ are the (unique with this property) lifts of the elements of $\mathbb{F}_p^{\times}$, which are all $p-1$st roots of unity.

If $p=2$, the only unit roots in $\mathbb{Z}_2^{\times}$ and even in $\mathbb{Z}_2^{\times}+p\mathcal O_D=1+p\mathcal O_D$ are $\pm 1$. Indeed, any unit root in this subring would be of $p$-power order, thus it is enough to show that the subring does not contain primitive 4th roots of unity. Assume that $\lambda$ is a primitive 4th root of unity. Then $\lambda^2=-1\in 1+1\cdot 2+\Pi p \mathcal O_D$. An easy calculation shows that such $\lambda$ do not exist in $1+p\mathcal O_D$.
\end{rem}

The proof of Theorem \ref{thmgenauRZnp} parallels the proof of Theorem \ref{thmgenautRZ}. Thus, we mainly indicate the differences and similarities. We still consider the isocrystal $(N,F)$ of $\mathbb{X}$, and the basis $e_i,f_i$ with $1\leq i\leq g$ chosen as before, just omitting the polarization/symplectic pairing.

Let us review some results from \cite{modpdiv}. By \cite[Lemma 4.6]{modpdiv} the locus $\mathcal M^{{\rm np},\circ}_g$ where the $a$-invariant of the universal $p$-divisible group is 1 is open and dense. By \cite[Thm.~B]{modpdiv}, the group $J^{{\rm np}}_b(\mathbb{Q}_p)\cong \GL_g(D)$ of self-quasi-isogenies of $\mathbb{X}$ acts transitively on the set of irreducible components of $\mathcal M^{{\rm np}}_g$ and of $\mathcal M^{{\rm np},\circ}_g$. The irreducible components are in bijection with $\tau$-stable lattices where $\tau=F^2p^{-1}$. In particular, to show Theorem \ref{thmgenauRZnp}, it is enough to show the analogous statement on some fixed irreducible component of $\mathcal M^{{\rm np},\circ}_g$. We consider the $\tau$-stable lattice $\Lambda_0$ from the previous sections and denote the corresponding irreducible component by $\mathcal C_{\Lambda_0}^{{\rm np}}$. Elements of $\mathcal M^{{\rm np}}_g(\overline{\mathbb{F}}_p)$ correspond to Dieudonn\'e lattices in $N$. A Dieudonn\'e lattice $M\subseteq N$ corresponds to a point of $\mathcal C_{\Lambda_0}^{{\rm np}}$ if and only if $a(M)=1$ and $M_{\tau}=\Lambda_0$ where $M_{\tau}$ is as before the smallest $\tau$-stable lattice containing $M$. From \cite[Lemma 4.7]{modpdiv} we obtain the following non-polarized analogue of Proposition \ref{propcondai}.
 
 \begin{prop}
 Let $v=\sum_{i=1}^g\sum_{l\geq 0} [a_{i,l}]\Pi^l X_i\in\Lambda_0$. Then $M:=\mathcal D v$ corresponds to a point of $\mathcal C_{\Lambda_0}^{{\rm np}}$ if and only if the $a_{i,0}\in\overline{\mathbb{F}}_p$ are linearly independent over $\mathbb{F}_{p^2}$. 
 
 Conversely, every Dieudonn\'e lattice corresponding to a point of $\mathcal C_{\Lambda_0}^{{\rm np}}$ is as $\mathcal D$-module generated by some (non-unique) $v$ as above.
 \end{prop}
 
As in the polarized case, we obtain that every $j\in J^{{\rm np}}(\mathbb{Q}_p)$ that is an automorphism of some $M$ with $M_{\tau}=\Lambda_0$ is also an automorphism of $\Lambda_0$. As in the polarized case, Lemma \ref{lemmaKY63} together with Remark \ref{remscalarZp} now imply that Theorem \ref{thmgenauRZnp} follows from the next claim.
\begin{claim}\label{keyclaimnp}
Let $g\geq 5$. For $s=2$ and all $p$, and for $s=3$ if $p=2$, there is a non-empty subscheme $Y_s$ of $\mathcal C_{\Lambda_0}$ such that for every $M\subset N$ corresponding to an $\overline{\mathbb F}_p$-valued point of $Y_s$ and every $j\in J^{{\rm np}}_b(\mathbb Q_p)$ stabilizing $\Lambda_0$ and $M$, we have that $j$ induces on $\Lambda_0/\Pi^s\Lambda_0$ a multiple of the identity. This factor is in $\mathbb{Z}_p^{\times}$ if $s=1$ or $2$ and in $\mathbb{Z}_p^{\times}+p\mathbb{Z}_{p^2}$ if $s=3$. 
\end{claim}

Most of the proof of this claim follows from the proof of Claim \ref{keyclaim}. Indeed, there we even showed that there are such lattices $M\subseteq \Lambda_0$ which are in addition self-dual. On the other side, we only considered $j\in J(\mathbb{Q}_p)\subseteq J^{{\rm np}}_b(\mathbb Q_p)$. However, this condition is only used in two places: For the first time in the last few lines of Section \ref{sec41}, to show that the $\sigma$-stable scalar $j_1$ is in fact $\pm 1$. We do not have that assertion in our case, and can instead multiply $j$ by $[j_1]^{-1}\in\mathbb{Z}_p^{\times}$ to reduce to the case $j_1=1$. Compatibility of $j$ with the polarization is used a second time at the beginning of Section \ref{sec43} to show that it is enough to prove that $j_3$ is a scalar. But in our case, being a scalar is already all we need to show for the claim.

Altogether we obtain Claim \ref{keyclaimnp} and thus Theorem \ref{thmgenauRZnp}. \qed


\bibliography{references}

\end{document}